\documentclass{amsart}
\usepackage{xypic}

\usepackage{color}

\title{On the functor $\ell^2$}
\author{Chris Heunen}
\address{Department of Computer Science, University of Oxford, Wolfson
  Building, Parks Road, OX1 3QD, Oxford, UK}
\email{heunen@cs.ox.ac.uk}
\date{\today}
\thanks{The author was supported by U.S. Office of Naval Research Grant Number
N000141010357, and would like to thank Samson Abramsky, Rasmus Bentmann, John Bourke, Bart Jacobs, Peter Hines, Robin Kaarsgaard, Martti Karvonen, Prakash Panangaden, Ji{\v{r}}{\'{\i}} Rosicky, Norbert Schuch, and Jamie Vicary for encouragement and pointers.}

\numberwithin{equation}{section}
\theoremstyle{plain}
\newtheorem{theorem}[equation]{Theorem}
\newtheorem{lemma}[equation]{Lemma}
\newtheorem{proposition}[equation]{Proposition}
\newtheorem{corollary}[equation]{Corollary}
\newtheorem*{corollary*}{Corollary}
\theoremstyle{definition}
\newtheorem{definition}[equation]{Definition}

\newtheorem{note}[equation]{}

\newcommand{\after}{\circ}
\newcommand{\blank}{\ensuremath{\underline{\phantom{n}}}}
\newcommand{\cat}[1]{\ensuremath{\mathbf{#1}}}
\newcommand{\Cat}[1]{\ensuremath{\mathbf{#1}}}
\newcommand{\id}[1][]{\ensuremath{\mathrm{id}_{#1}}}

\newcommand{\op}{\ensuremath{^{\mathrm{op}}}}
\newcommand{\inprod}[2]{\ensuremath{\langle #1\,|\,#2 \rangle}}
\newcommand{\field}[1]{\ensuremath{\mathbb{#1}}}
\newcommand{\tensor}{\ensuremath{\otimes}}
\newcommand{\norm}[2][]{\ensuremath{\|#2\|_{#1}}}
\newcommand{\dom}{\ensuremath{\mathrm{dom}}}

\renewcommand{\Im}{\ensuremath{\mathrm{Im}}}
\newcommand{\Dm}{\ensuremath{\mathrm{Dom}}}
\newcommand{\Set}{\Cat{Set}}
\newcommand{\Hilb}{\Cat{Hilb}}
\newcommand{\Hilbi}{\Cat{Hilb}_{\cong}}
\newcommand{\PInj}{\Cat{PInj}}
\newcommand{\one}{\cat{1}}
\newcommand{\two}{\cat{2}}
\newcommand{\three}{\cat{3}}
\newcommand{\ltwo}{\ensuremath{\ell^2}}
\newcommand{\ie}{\textit{i.e.}~}
\newcommand{\eg}{\textit{e.g.}~}
\newcommand{\xyline}[2][]{\ensuremath{\smash{\xymatrix@1#1{#2}}}}
\makeatother
 \newdir{ >}{{}*!/-7.5pt/@{>}}
\makeatletter

\begin{document}
\maketitle
\begin{abstract}
  We study the functor $\ltwo$ from the category of partial injections
  to the category of Hilbert spaces. The former category is finitely accessible, 
  and its homsets are algebraic domains; the latter category has conditionally algebraic domains for homsets.
  The functor preserves daggers, monoidal structures, enrichment, and
  various (co)limits, but has no adjoints. Up to unitaries, its direct
  image consists precisely of the partial isometries, but its
  essential image consists of all continuous linear maps between
  Hilbert spaces. 
\end{abstract}

I am delighted to dedicate this paper to Samson Abramsky, on the
occasion of his 60th birthday. 
Among all the wisdom he has imparted on me is this contradictory gem:
``Never solve a problem completely, or noone will have a reason to cite you''.  
My better nature gladly took some time off to let this paper follow his advice.

\section{Introduction}

The rich theory of Hilbert spaces underpins much of modern functional
analysis and therefore quantum
physics~\cite{reedsimon:functionalanalysis,kadisonringrose:operatoralgebras},
yet important parts of it have resisted categorical treatment. 
In any categorical analysis of a species of mathematical
objects, free objects of that kind play a significant role.
The important $\ltwo$--construction is in many ways the closest thing
there is to a free Hilbert space: if $X$ is a set, then
\[
  \ltwo(X) = \Big\{ \varphi \colon X \to \field{C} \;\Big|\; \sum_{x \in X}
  |\varphi(x)|^2 < \infty \Big\}
\]
is a Hilbert space, in fact the only one of its dimension up to isomorphism.
The $\ltwo$--construction can 
be made into a functor, if we take partial injections as morphisms
between the sets $X$, as first observed by Barr~\cite{barr:algebraicallycompact}.
Outside functional analysis, it also plays a historically important
role in the geometry of interaction (which has been noticed by many
authors; an incomplete list of references
includes~\cite{danosregnier:goi,abramsky:retracing,haghverdi:phd,haghverdiscott:goi,hines:phd}).

Explicit categorical properties of the $\ltwo$--construction are few
and far between in the  
literature. These notes gather and augment them in
a systematic study. Section~\ref{sec:hilb} starts with the category of Hilbert
spaces: it is self-dual, has two monoidal structures, and its homsets are
algebraic domains, but its enrichment and limit
behaviour is wanting. Section~\ref{sec:pinj} discusses the category of
partial injections, which is more well-behaved: it is also self-dual,
has two monoidal structures, and is enriched over algebraic domains;
moreover, it is finitely  
accessible. Section~\ref{sec:ltwo} introduces and studies the functor 
$\ltwo$ itself. It preserves the self-dualities, monoidal structures,
and enrichment. It also preserves (co)kernels and finite (co)products,
but not general (co)limits. Therefore it has no adjoints, and in that
sense does not provide free Hilbert spaces. It is faithful and
essentially surjective on objects. Section~\ref{sec:image} studies the
image of the functor $\ltwo$. Up to unitaries, its direct image
consists precisely of partial isometries. Remarkably, it is
essentially full, that is, its essential image is the whole category
of Hilbert spaces. 

Choice issues are lurking closely beneath the surface of these results. In
fact, $\ltwo(X)$ is not just a Hilbert space; it carries a priviledged
orthonormal basis. The functor $\ltwo$ is an equivalence between the
category of partial injections, and the category of Hilbert spaces
with a chosen orthonormal basis and morphisms preserving it. But the
latter class of morphisms is too restrictive: all
interesting applications of Hilbert spaces require a change of basis. 
Following the guiding thought ``a gentleman does not choose a
basis'', Section~\ref{sec:conclusion} suggests directions for further
research.

\section{The codomain}\label{sec:hilb}

\begin{definition}
  We are interested in the category $\Hilb$, whose objects are complex
  Hilbert spaces, and whose morphisms are continuous linear
  functions. 
\end{definition}

\begin{note}
  The category $\Hilb$ has a \emph{dagger}, that is, a contravariant
  involutive functor $\dag \colon \Hilb\op \to \Hilb$ that acts as the
  identity on objects. On a morphism $f \colon H \to K$ it is given by
  the unique adjoint $f^\dag \colon K \to H$ satisfying
  $\inprod{f(x)}{y}=\inprod{x}{f^\dag(y)}$. For example, an
  isomorphism $u$ is unitary when $u^{-1}=u^\dag$.
\end{note}

\begin{note}
  Furthermore, the usual tensor product of Hilbert spaces provides the
  category $\Hilb$ with symmetric monoidal structure. The monoidal
  unit is the 1-dimensional Hilbert space $\field{C}$. In fact,
  $\Hilb$ has \emph{dagger symmetric monoidal} structure, \ie $(f
  \tensor g)^\dag = f^\dag \tensor g^\dag$, and all coherence
  isomorphisms are unitaries.
\end{note}

\begin{note}
  Direct sums of Hilbert spaces provide the category $\Hilb$ with
  (finite) \emph{dagger biproducts}. That is, $H \oplus K$ is
  simultaneously a product and a coproduct, the projections are
  the daggers of the corresponding coprojections, and $(f \oplus
  g)^\dag = f^\dag \oplus g^\dag$. Similarly, the 0-dimensional Hilbert space 
  is a \emph{zero object}, \ie simultaneously initial and terminal.
\end{note}

\begin{note}
  Let us emphasize that we take continuous linear maps as morphisms
  between Hilbert spaces, rather than linear contractions. The
  category of Hilbert spaces with the latter morphisms is rather
  well-behaved, see \eg~\cite{adamekrosicky:locallypresentable}. 
  However, it is the former choice of morphisms that is of interest in
  functional analysis and quantum physics. Unfortunately it also
  reduces the limit behaviour of the category $\Hilb$, as the
  following lemma shows.
\end{note}

\begin{lemma} 
\label{limitsinhilb}
  The category $\Hilb$:
  \begin{enumerate}
  \item[(i)] has (co)equalizers;
  \item[(ii)] does not have infinite (co)products;
  \item[(iii)] does not have directed (co)limits.
  \end{enumerate}
\end{lemma}
\begin{proof}
  Part (i) holds because $\Hilb$ is enriched over abelian groups and
  has kernels~\cite{heunenjacobs:dagkercat}. For (ii), consider the
  following counterexample. Define an $\field{N}$-indexed 
  family $H_n = \field{C}$ of objects of $\Hilb$. Suppose the
  family $(H_n)$ had a coproduct $H$ with coprojections $\kappa_n
  \colon H_n \to H$. Define $f_n\colon H_n \to \field{C}$ by $f_n(z) =
  n \cdot \norm{\kappa_n} \cdot z$. These are bounded maps, since
  $\norm{f_n}=n \cdot \norm{\kappa_n}$. Then for all $n \in \field{N}$
  the norm of the cotuple $f\colon H \to \field{C}$ of $(f_n)$ must satisfy  
  \[
    n \cdot \norm{\kappa_n} 
    = \norm{f_n} 
    = \norm{f \after \kappa_n}
    \leq \norm{f} \cdot \norm{\kappa_n},
  \]
  so that $n \leq \norm{f}$. This contradicts the boundedness and
  hence continuity of $f$. Finally, part (iii) follows from (ii)
  and~\cite[IX.1.1]{maclane:categories} 
\end{proof}

\begin{note}
\label{hilbconditionallycomplete}
  Despite the previous lemma, $\Hilb$ is \emph{conditionally
  (co)complete}, in the sense that it does have objects that partially
  obey the universal property of infinite (co)products: for a family
  $H_i$ of Hilbert spaces, 
  \[
    H = \big\{ (x_i) \in \prod_i H_i \mid
           \sum_i \|x_i\|^2 < \infty \big\}.
  \]
  is a well-defined Hilbert space under the inner
  product $\inprod{ (x_i) }{ (y_i) } = \sum_i
  \inprod{x_i}{y_i}$~\cite{kadisonringrose:operatoralgebras}. 
  The evident morphisms $\pi_i \colon H \to H_i$ satisfy
  $\pi_i \circ \pi_i^\dag = \id$ and $\pi_i \circ \pi_j^\dag = 0$ when
  $i \neq j$.
  A cone $f_i \colon K \to H_i$ allows a unique well-defined morphism
  $f\colon K \to H$ satisfying $\pi_i \after f = f_i$ if and only if $\sum_i
  \|f_i\|^2 < \infty$. 
  Note, however, that the cone $(\pi_i)$ itself does not satisfy this
  condition.
  In this sense, $\ltwo(X)$ is the conditional coproduct of $X$ many
  copies of $\field{C}$.
\end{note}

\begin{note}\label{note:bimorphism}
  A similar phenomenon occurs for simpler types of
  (co)limits. Monomorphisms in $\Hilb$ are precisely the 
  injective morphisms, and epimorphisms are precisely those
  morphisms with dense range~\cite[A.3]{heunen:embedding}. 
  Not every monic epimorphism is an isomorphism.
  For example, the morphism $f \colon \ltwo(\field{N}) \to \ltwo(\field{N})$
  defined by $f(\varphi)(n) = \frac{1}{n} \varphi(n)$ is injective, self-adjoint,
  and hence also has dense image. But it is not surjective, as the
  vector $\varphi \in \ltwo(\field{N})$ determined by
  $\varphi(n)=\tfrac{1}{n}$ is not in its range.
  Monic epimorphisms are called \emph{bimorphisms}.
\end{note}

\begin{note}
\label{enrichmenthilb}
  If $f,g \colon H \to K$ are morphisms in $\Hilb$, then so are $f+g$
  and $z f$ for $z \in \mathbb{C}$. Because composition respects these
  operations, $\Hilb$ is enriched over complex vector spaces. In
  general, the homsets are not Hilbert spaces
  themselves~\cite{abramskyblutepanangaden:nuclearideals}, so $\Hilb$
  is not enriched over itself, and hence not Cartesian closed. At any rate, there 
  is another way to structure the homsets of $\Hilb$, which is of more interest
  here.  Say $f \leq g$ when $\ker(f)^\perp \subseteq \ker(g)^\perp$
  and $f(x)=g(x)$ for $x \in \ker(f)^\perp$. The following proposition
  shows that this makes all homsets into \emph{conditional algebraic
  domains}, but that this is not respected by composition. 
  Here, by conditional algebraic domain we mean an algebraic domain~\cite{abramskyjung:domaintheory}, except that it does not need to have all directed suprema, but only bounded ones, in the same sense as~\ref{hilbconditionallycomplete}: more precisely, every is a directed family of parallel morphisms $f_i$ in $\cat{Hilb}$ for which $\|f_i\|$ converges has a supremum.
  This is closely related
  to~\cite[2.1.4]{cockettlack:restriction1}, but $\Hilb$ is not a
  restriction category in the sense of that paper: setting
  $\overline{f}$ to be the projection onto $\ker(f)^\perp$ does not
  satisfy $\overline{f} \overline{g} = \overline{g} \overline{f}$.
\end{note}

\begin{proposition}
\label{domainshilb}
  All homsets in the category $\Hilb$ are conditional algebraic domains, but
  composition is not monotone.
\end{proposition}
\begin{proof}
  The least upper bound of a directed family $f_i$ is given by
  continuous extension to the closure of $\bigcup_i \ker(f_i)^\perp$;
  this makes all homsets into directed-complete partially ordered
  sets. If $f \leq \bigvee_i f_i$ always implies $f \leq f_i$ for some
  $i$, then $\ker(f)^\perp$ must have been finite-dimensional; thus
  morphisms $f$ with $\dim(\ker(f)^\perp)<\infty$ are the compact
  elements. It is now easy to see that any morphism is the directed
  supremum of compact ones below it, making all homsets into algebraic
  domains.  

  Now consider composition. First suppose that $f \leq f'$ and $g \leq g'$. 
  If $x \in \ker(f)$, then clearly $gf(x) = 0$. 
  If $x \in \ker(f)^\perp$, then $f(x)=f'(x)$, 
  so $g'f'(x)=0$ implies $f(x)\in \ker(g') \subseteq \ker(g)$.
  Because we may write $\dom(gf)=\ker(f) \oplus \ker(f)^\perp$, we
  conclude 
  $\ker(gf)^\perp \subseteq
  \ker(g'f')^\perp$. But unless $f(\ker(gf)^\perp) \subseteq
  \ker(g)^\perp$, it need not be the case that $gf$ equals $g'f'$ on
  $\ker(gf)^\perp$. For an explicit counterexample, let 
  \[
    f=f'=\begin{pmatrix} 1 & 1 \\ 0 & 1 \end{pmatrix},
    \qquad
    g=\begin{pmatrix} 1 & 0 \\ 0 & 0 \end{pmatrix},
    \qquad
    g'=\begin{pmatrix} 1 & 0 \\ 0 & 1 \end{pmatrix}.
  \]
  Then $f\leq f'$ and $g \leq g'$. But $\ker(gf)^\perp =
  \{ \left(\begin{smallmatrix} x \\ -x \end{smallmatrix}\right) \mid x \in \mathbb{C} \}^\perp =
  \{ \left(\begin{smallmatrix} x \\ x \end{smallmatrix}\right) \mid x \in \mathbb{C} \}$,
  and $gf \left(\begin{smallmatrix} x \\ x \end{smallmatrix}\right) =
  \left(\begin{smallmatrix}2x \\  0 \end{smallmatrix}\right) \neq
  \left(\begin{smallmatrix} 2x \\ x \end{smallmatrix}\right) =
  g'f' \left(\begin{smallmatrix} x \\ x \end{smallmatrix}\right)$, so
  $gf \not\leq g'f'$. 
\end{proof}

\section{The domain}\label{sec:pinj}

\begin{definition}
  A \emph{partial injection} is a partial function that is injective,
  wherever it is defined. More precisely, it(s graph) is a
  relation $R \subseteq X \times Y$ such that for each $x$ there is
  at most one $y$ with $(x,y) \in R$, and for each $y$ there is
  at most one $x$ with $(y,x) \in R$. Sets and partial injections
  form a category $\PInj$ under composition of relations $S \circ R =
  \{(x,z) \mid \exists y \colon (x,y) \in R, (y,z) \in S \}$.
\end{definition}

\begin{note}
\label{spans}
  Notationally, a partial injection $f \colon X \to Y$ can be
  conveniently represented as a span
  $(\xyline{X & F \ar@{ >->}|-{f_1}[l] \ar@{ >->}|-{f_2}[r] & Y})$
  of monics in $\Set$.
  Here, $f_1$ is (the inclusion of) the domain of definition of $f$,
  and $f_2$ is its (injective) action on that domain.
  Composition in this representation is by pullback.
  We will also write $\Dm(f)=f_1(F)$ for the domain of definition, and
  $\Im(f)=f_2(F)$ for the range of $f$.

  If it wasn't already, the span notation immediately makes it clear
  that $\PInj$ is a \emph{dagger} category: 
  $(\xyline{X & F \ar@{ >->}|-{f_1}[l] \ar@{ >->}|-{f_2}[r] & Y})^\dag
  = (\xyline{Y & F \ar@{ >->}|-{f_2}[l] \ar@{ >->}|-{f_1}[r] & X})$.
\end{note}

\begin{note}
  The category $\PInj$ has two dagger \emph{symmetric monoidal}
  structures. The first one, that we denote by $\tensor$, acts as the
  Cartesian product on objects. Because the Cartesian product of
  injections is again injective, $\tensor$ is well-defined on
  morphisms of $\PInj$ as well. The monoidal unit is a singleton set
  $\one$. Notice that $\tensor$ is not a product, and hence not a
  coproduct either. 

  The second dagger symmetric monoidal structure on $\PInj$, denoted
  by $\oplus$, is given by disjoint union on objects. It is easy to
  see that a disjoint union of injections is again injective, making
  $\oplus$ well-defined on morphisms of $\PInj$. The monoidal unit is
  the empty set. Notice that $\oplus$ is not a coproduct, and hence
  not a product either. 
\end{note}

\begin{lemma}
\label{limitsinpinj}
  The category $\PInj$:
  \begin{enumerate}
  \item[(i)] has (co)equalizers;
  \item[(ii)] has a zero object;
  \item[(iii)] does not have finite (co)products;
  \end{enumerate}
\end{lemma}
\begin{proof}
  The equalizer of $f,g \colon X \to Y$ is the inclusion of
  \[
    \big\{ x \in X \mid x \not \in (\Dm(f) \cup \Dm(g))
    \vee \big(x \in (\Dm(f) \cap \Dm(g)) \wedge f(x)=g(x) \big) \big\} 
  \]
  into $X$. The empty set is a zero object in $\PInj$.

  Towards (iii), notice that if $(\xyline{X \ar|(.33){\kappa_X}[r] &
  X+Y & Y \ar|(.33){\kappa_Y}[l]})$ were a coproduct in $\PInj$, then
  one must have $\Dm(\kappa_X)=X$, $\Dm(\kappa_Y)=Y$ and
  $\Im(\kappa_X) \cap \Im(\kappa_Y) = \emptyset$, because otherwise
  unique existence of mediating morphisms is violated. Hence any
  coproduct must contain the disjoint union of $X$ and $Y$. 
  Let $f \colon X \to Z$ and $g \colon Y \to Z$ be
  any morphisms. Then a mediating morphism $m \colon X+Y \to Z$ has to
  satisfy $m(x)=f(x)$ for $x \in \Dm(f)$ and $m(y)=g(y)$ for $y \in
  \Dm(g)$. But such an $m$ is not unique, unless $\Dm(f)=X$ and
  $\Dm(g)=Y$. In fact, it is not even a partial injection unless
  $\Im(f) \cap \Im(g) = \emptyset$. We conclude that $\PInj$ does
  not have binary (co)products.
\end{proof}

\begin{note}
\label{infpaths}
  In fact, part (ii) of the previous lemma follows from the
  existence of directed colimits, which we now work towards. Recall
  that a category has directed colimits if and only if it has colimits
  of chains, \ie colimits of well-ordered
  diagrams~\cite[Corollary~1.7]{adamekrosicky:locallypresentable}.
  Observe that for a chain $D \colon I \to \PInj$,
  if $c_i \colon D(i) \to X$ is a cocone on $D$, 
  then $\Dm(c_i) \subseteq \Dm(D(i\leq j))$ for all $j \geq i$.
  To see this, notice that $c_i = c_j \after D(i \leq j)$ since $c_i$
  is a cocone, and therefore
  \[
      \Dm(c_i) 
    = \Dm(c_j \after D(i \leq j))
    \subseteq \Dm(D(i \leq j)).
  \]  
  This observation suggests that the colimit of a well-ordered diagram
  in $\PInj$ should consist of all `infinite paths'. The
  following proposition shows that this is indeed a colimit.
\end{note}

\begin{proposition}
\label{pinjdircolim}
  The category $\PInj$ has directed colimits.
\end{proposition}
\begin{proof}
  Let $D \colon I \to \PInj$ be a chain. Define
  \[
    X = \{ x \in \coprod_i D(i) \mid 
        \forall_{j\geq i}[x \in \Dm(D(i \leq j))]\} / \sim,
  \]
  where the coproduct is taken in $\Set$, and the equivalence
  relation $\sim$ is generated by $x \sim D(i \leq j)(x)$ for all $i
  \leq j$ in $I$ and $x \in \Dm(D(i \leq j))$. 
  For $i \in I$, define $c_i \colon D(i) \to X$ by
  \[
    \Dm(c_i) = \{ x \in D(i) \mid
        \forall_{j\geq i}[x \in \Dm(D(i \leq j))]\},
  \]
  and $c_i(x) = [x]$.

  First of all, let us show that the $c_i$ form a cocone. One has:
  \begin{align*}
    & \Dm(c_j \after D(i \leq j)) \\
    & = \{ x \in D(i) \mid 
        x \in \Dm(D(i \leq j)) \wedge D(i \leq j)(x) \in \Dm(c_j) \} \\
    & = \{ x \in D(i) \mid
        x \in \Dm(D(i \leq j)) \wedge \forall_{k\geq j}[ D(i \leq
        j)(x) \in \Dm(D(j \leq k)) ] \}.
  \end{align*}
  The well-orderedness of $I$ implies that
  \[
    \forall_{k \geq i}[P(k)] \Leftrightarrow \forall_{k \geq j}[P(k)]
    \wedge P(j)
  \]
  for any property $P$ on the objects of $I$, whence
  \[
      \Dm(c_j \after D(i \leq j))
    = \{ x \in D(i) \mid 
                \forall_{k \geq i}[ x \in \Dm(D(i\leq k)) ] \}
    = \Dm(c_i).
  \]
  Moreover $c_j \after D(i \leq j)(x) = [D(i \leq j)(x)] = [x] = c_i(x)$
  for $x \in \Dm(c_i)$, by definition of the equivalence relation.

  Next, we show that $c_i$ is universal. Let $d_i \colon D(i) \to Y$
  be any cocone, and define $m \colon X \to Y$ by
  \[
    \Dm(m) = \{ [x] \mid x \in \Dm(d_i) \}
  \]
  and $m([x]) = d_i(x)$ for $x \in D(i)$; this is well-defined since
  $d_i$ is a cocone. Then
  \begin{align*}
        \dom(m \after c_i)
    & = \{ x \in D(i) \mid x \in \Dm(c_i) \wedge m_i(x) \in \Dm(m) \} \\
    & = \{ x \in D(i) \mid \forall_{j \geq i}[ x \in \Dm(D(i \leq j))] 
        \wedge x \in \Dm(d_i) \} \\
    & = \Dm(d_i)
  \end{align*}
  by~\ref{infpaths}, 
  and $m \after c_i(x) = m([x]) = d_i(x)$ for $x \in D(i)$. Thus $m
  \after c_i = d_i$, so $m$ is indeed a mediating morphism. 

  Finally, if $m' \colon X \to Y$ satisfies $m \after c_i = d_i$, then
  it follows from the above considerations that $\Dm(m')=\Dm(m)$ and
  $m'(x)=m(x)$ for $x \in \Dm(m)$. Hence $m$ is the unique mediating
  morphism. 
\end{proof}

\begin{note}
\label{finpres}
  Recall that an object $X$ in a category $\cat{C}$ is called
  \emph{finitely presentable} when the hom-functor
  $\cat{C}(X,-) \colon \cat{C} \to \Set$ preserves directed colimits.
  Explicitly, this means that for any directed poset $D \colon I \to
  \cat{C}$, any colimit cocone $d_i \colon D(i) \to Y$ and any
  morphism $f \colon X \to Y$, there are $j \in I$ and a morphism $g
  \colon X \to D(j)$ such that $f = d_j \after g$. Moreover, this
  morphism $g$ is essentially unique, in the sense that if $f=d_i
  \after g = d_i \after g'$, then $D(i \to i') \after g = D(i \to i')
  \after g'$ for some $i' \in I$. 
  \[\vcenter{\hbox{\xymatrix{
          D(i) \ar[r] \ar_-{d_i}[dr] 
          & D(i') \ar[r] \ar_-{d_{i'}}[d] 
          & D(i'') \ar[r] \ar_-{d_{i''}}[dl]
          & \cdots \ar[r] 
          & D(j) \ar[r] \ar_-{d_j}[dlll]
          & \cdots \\
          & Y & & & X \ar^-{f}[lll] \ar@{-->}_-{g}[u]
        }}}
  \]
  A category is called \emph{finitely
  accessible}~\cite{adamekrosicky:locallypresentable} when it has
  directed colimits and every object is a directed colimit of finitely
  presentable objects. 
\end{note}

\begin{lemma}
  A set is finitely presentable in $\PInj$ if and only if it is
  finite.  
\end{lemma}
\begin{proof}
  The only thing, in the situation of~\ref{finpres} with $X$ finite,
  is to notice that if a partial injection $g$ is to exist, we must
  have $\Dm(g)=\Dm(f)$. The rest follows
  from~\cite[1.2.1]{adamekrosicky:locallypresentable}. 
\end{proof}

\begin{theorem}
\label{pinjlacc}
  The category $\PInj$ is finitely accessible.
\end{theorem}
\begin{proof}
  It suffices to prove that every set in $\PInj$ is a directed colimit
  of finite ones. But that is easy: $X$ is the colimit of the directed
  diagram consisting of its finite subsets.
\end{proof}

\begin{definition}
  An \emph{inverse category} is a category $\cat{C}$ in which every
  morphism $f \colon X \to Y$ allows a unique morphism $f^\dag \colon Y
  \to X$ satisfying $f=ff^\dag f$ and $f^\dag=f^\dag f f^\dag$.
  Equivalently, it is a dagger category satisfying
  $f=ff^\dag f$ and $pq=qp$ for idempotents $p,q \colon X
  \to X$.
  The proof of equivalence of these two statements is the same as for
  inverse semigroups
  (see~\cite[Theorem~1.1.3]{lawson:inversesemigroups} or~\cite[Theorem~2.20]{cockettlack:restriction1}).
  Inverse categories are a special case of \emph{restriction
  categories}~\cite{cockettlack:restriction1}. 

  The category $\PInj$ is an inverse category under its dagger
  (see~\ref{spans}). The following 
  categorification of the Wagner--Preston 
  theorem~\cite[Theorem~1.5.1]{lawson:inversesemigroups} 
  shows that it is in fact a representative one. See
  also~\cite[3.4]{cockettlack:restriction1}. 
\end{definition}  

\begin{proposition} \cite{kastl:inverse}
  Any locally small inverse category $\cat{C}$ allows
  a faithful embedding $F \colon \cat{C} \to \PInj$ that preserves daggers.
\end{proposition}

\begin{note}
\label{domainspinj}
  Like any inverse category, the homsets of $\PInj$ carry a natural
  partial order: $f \leq g$ when $f=gf^\dag f$. Concretely, $f \leq g$
  means $\Dm(f) \subseteq \Dm(g)$ and $f(x)=g(x)$ for $x \in \Dm(f)$.  
  It is easy to see that this makes homsets into directed-complete
  partially ordered sets, with $\Dm(\bigvee_i f_i) = \bigcup_i
  \Dm(f_i)$ for a directed family of morphisms $f_i \colon X \to Y$.
  In fact, as in Proposition~\ref{domainshilb}, homsets are algebraic domains:
  any partial injection is the supremum of compact ones below it,
  which are those partial injections with finite domain. Moreover,
  composition respects these operations. Thus $\PInj$ is enriched in
  algebraic domains. This is a satisfying reflection of
  Theorem~\ref{pinjlacc} on the level of homsets.
\end{note}

\section{The functor}\label{sec:ltwo}

\begin{definition}
\label{def:ltwo}
  There is a functor $\ltwo \colon \PInj \to \Hilb$, acting on a set
  $X$ as 
  \[
    \ltwo(X) = \{ \varphi \colon X \to \field{C} \mid \sum_{x \in X}
    |\varphi(x)|^2 < \infty \},
  \]
  which is a well-defined Hilbert space under the inner product
  $\inprod{\varphi}{\psi} = \sum_{x \in X} \overline{\varphi(x)}
  \psi(x)$. The action on morphisms sends a partial injection 
  $(\xyline{X & F \ar@{ >->}|-{f_1}[l] \ar@{ >->}|-{f_2}[r] & Y})$
  to the linear function $\ltwo f \colon \ltwo(X) \to \ltwo(Y)$
  determined informally by $\ltwo f = (\blank) \after
  f^\dag$. Explicitly, 
  \[
    (\ltwo f)(\varphi)(y) = \sum_{x \in f_2^{-1}(y)} \varphi(f_1(x)).
  \]
\end{definition}

\begin{note}
  In verifying that $\ltwo f$ is indeed a well-defined morphism of
  $\Hilb$, it is essential that $f$ is a (partial) injection. 
  \begin{align*}
        \sum_{y \in Y} \big|(\ltwo f)(\varphi)(y)\big|^2
    & = \sum_{y \in Y} \big| \sum_{x \in f_2^{-1}(y)} \varphi(f_1(x)) \big|^2 
      \leq \sum_{y \in Y} \sum_{x \in f_2^{-1}(y)} |\varphi(f_1(x))|^2 \\
    & = \sum_{x \in F} |\varphi(f_1(x))|^2 
      \leq \sum_{x \in X} |\varphi(x)|^2 
      < \infty.
  \end{align*}
  That this breaks down for functions $f$ in general, instead of
  (partial) injections, was first noticed
  in~\cite{barr:algebraicallycompact}, and further studied in~\cite{haghverdiscott:goi}. 
  That is, $\ltwo$ is well-defined on the category of sets and partial injections; on the
  category of finite sets and functions; but not on the category of sets and functions; nor
  on the category of finite sets and relations.
  Functoriality of $\ltwo$ is easy to verify.
\end{note}

\begin{note}
  The following calculation shows that the $\ltwo$ functor preserves
  daggers. For a partial injection
  $(\xyline{X & F \ar@{ >->}|-{f_1}[l] \ar@{ >->}|-{f_2}[r] & Y})$,
  $\varphi \in \ltwo(X)$ and $\psi \in \ltwo(Y)$: 
  \begin{align*}
        \inprod{(\ltwo f)(\varphi)}{\psi}_{\ltwo(Y)}
    & = \sum_{y \in Y} \overline{(\ltwo f)(\varphi)(y)} \cdot \psi(y) 
      = \sum_{y \in Y} \sum_{x \in f_2^{-1}(y)} \overline{\varphi(f_1(x))}
                      \cdot \psi(y) \\
    & = \sum_{x \in F} \overline{\varphi(f_1(x))} \cdot \psi(f_2(x)) 
      = \sum_{x \in X} \sum_{x' \in f_1^{-1}(x)} \overline{\varphi(x)}
                      \cdot \psi(f_2(x')) \\
    & = \sum_{x \in X} \overline{\varphi(x)} \cdot 
                      (\sum_{x' \in f_1^{-1}(x)} \psi(f_2(x')))
      = \inprod{\varphi}{\ltwo(f^\dag)(\psi)}_{\ltwo(X)}.
  \end{align*}
\end{note}

\begin{note}
  The functor $\ltwo$ preserves the tensor product $\tensor$, \ie it
  is symmetric (strong) monoidal. There is a canonical isomorphism
  $\field{C} \cong \ltwo(\one)$. The required natural
  morphisms $\ltwo(X) \tensor \ltwo(Y) \to \ltwo(X \tensor Y)$ are
  given by mapping $(\varphi,\psi)$ to the function $(x,y) \mapsto
  \varphi(x)\psi(y)$. That there are inverses is seen when one
  realizes that $\ltwo(X \tensor Y)$ is the Cauchy-completion of the
  set of functions $X \times Y \to \field{C}$ with finite support. The
  required coherence diagrams follow easily.
\end{note}

\begin{note}
  Also, the $\ltwo$ functor is symmetric (strong) monoidal with
  respect to $\oplus$. There is a canonical isomorphism between the
  0-dimensional Hilbert space and the set $\ltwo(\emptyset)$
  consisting only of the empty function. The natural morphisms
  $\ltwo(X) \oplus \ltwo(Y) \to \ltwo(X \oplus Y)$ map
  $(\varphi,\psi)$ to the cotuple $[\varphi,\psi] \colon X \oplus Y
  \to \field{C}$. One sees that these are isomorphisms by recalling
  that $\ltwo(X \oplus Y)$ is the closure of the span of the Kronecker
  functions $\delta_x$ and $\delta_y$ for $x \in X$ and $y \in Y$, on
  which the inverse acts as the appropriate coprojection. Coherence
  properties readily follow.
\end{note}

\begin{note}
  From the description of the structure of homsets in $\PInj$ and
  $\Hilb$ as algebraic domains in~\ref{domainspinj}
  and~\ref{enrichmenthilb}, respectively, it is clear that the functor
  $\ltwo$ preserves this enrichment:  $\ltwo(\bigvee_i f_i) = \bigvee_i \ltwo f_i$ if $f_i
  \colon X \to Y$ is a directed family of morphisms in $\PInj$.  
  See also~\cite[Theorem~13]{hines:oracle}.
\end{note}

\begin{note}
\label{preserveequalizers}
  The functor $\ltwo$ preserves (co)kernels and finite (co)products
  (because $\PInj$ has very few of the latter).
  But it follows from Lemma~\ref{limitsinhilb}(iii) and
  Proposition~\ref{pinjdircolim} that 
  $\ltwo$ cannot preserve arbitrary (co)limits. For an
  explicit counterexample to preservation of equalizers, take
  $X=\{0,1\}$, $Y=\{a\}$, and let $f,g\colon X 
  \to Y$ be the partial injections $f=\{(0,a)\}$ and
  $g=\{(1,a)\}$. Their equaliser in $\Cat{PInj}$ is $\emptyset$. But
  \begin{align*}
       \mathrm{eq}(\ell^2(f), \ell^2(g)) 
    &= \{ \varphi \in \ell^2(X) \mid \ell^2(f)(\varphi) =
                                     \ell^2(g)(\varphi) \} \\
    &= \Big\{ \varphi \in \ell^2(X) \mid \forall_{y \in Y}.\,
                      \sum_{u \in f_2^{-1}(y)} \varphi(f_1(u)) =
                      \sum_{v \in g_2^{-1}(y)} \varphi(g_1(v)) \Big\} \\
    &= \{ \varphi \colon  \{0,1\} \to \field{C} \mid \varphi(0) = \varphi(1) \} 
    \cong \field{C}.
  \end{align*}
  Hence $\mathrm{eq}(\ell^2(f),\ell^2(g)) \cong \field{C} \not\cong
  \{\emptyset\} = \ell^2(\mathrm{eq}(f,g))$. 
\end{note}

\begin{corollary}
\label{noadjoints}
  The functor $\ltwo \colon \PInj \to \Hilb$ has no adjoints. 
\end{corollary}
\begin{proof}
  If $\ltwo$ had an adjoint, it would preserve (co)limits,
  contradicting~\ref{preserveequalizers}.
\end{proof}

\begin{note}
\label{eso}
  The functor $\ltwo$ is clearly faithful. It is also essentially
  surjective on objects: every Hilbert space $H$ has an orthonormal
  basis $X$, so $H \cong \ltwo(X)$. 
  It cannot be full because of~\ref{noadjoints}, but it does reflect
  isomorphisms: if $\ltwo f$ is invertible, so is $f$. 
\end{note}

\begin{note}
\label{orthobasis}
  If $X$ is a set, $\ell^2(X)$ is not just a Hilbert space; it
  comes equipped with a chosen orthonormal basis (given by the
  Kronecker functions $\delta_x \in \ell^2(X)$ for $x \in X$).
  Hence we could think of $\ltwo$ as a functor to a category of
  Hilbert spaces $H$ with a priviledged orthonormal basis $X \subseteq
  H$. If we choose as morphisms $(H,X) \to (K,Y)$ those continuous
  linear $f \colon H \to K$ satisfying $f(X) \subseteq Y$ and
  $f f^\dag f = f$, then the functor $\ltwo$ in fact becomes
  (half of) an equivalence of categories~\cite[4.3]{abramskyheunen:hstar}.
\end{note}

\begin{note}
  Lemma~\ref{noadjoints} showed that $\ltwo(X)$ is not the free
  Hilbert space on $X$, at least not in the categorically accepted
  meaning. It also makes precise the intuition
  that `choosing bases is unnatural': the
  functor $\ltwo \colon \PInj \to \Hilb$ cannot have a 
  (functorial) converse, even though one can choose an orthonormal
  basis for every Hilbert space.  

  It is perhaps also worth mentioning that $\ltwo$ is not a fibration
  in the technical sense of the word, not even a nonsplit or noncloven
  one, as the reader might perhaps think; Cartesian liftings in
  general do not exist because `choosing bases is unnatural'. 
\end{note}

\section{The image}\label{sec:image}

\begin{note}
  The choice of morphisms in~\ref{orthobasis} is quite strong, and
  does not capture all morphisms of interest to quantum physics. From
  that point of view, one would at least like to relax to \emph{partial 
  isometries}: morphisms $i$ of Hilbert spaces that
  satisfy $i i^\dag i = i$. Equivalently, the restriction of $i$ to the 
  orthogonal complement of its kernel is an isometry. The following
  proposition proves that, up to isomorphisms, the direct image of the
  functor $\ltwo$ consists precisely of partial isometries.
\end{note}

\begin{definition}
  For a category $\cat{C}$, denote by $\cat{C}_{\cong}$ the category
  with the same objects as $\cat{C}$ whose morphisms are the
  bimorphisms of $\cat{C}$.

  The category $\Hilbi$ carries two dagger symmetric monoidal structures:
  $\oplus$ and $\otimes$. Because having (co)limits only depends on a
  skeleton of the specifying diagram, $\Hilbi$ does not have (co)equalizers,
  nor (finite) (co)products, but does have directed (co)limits.
\end{definition}

\begin{proposition}
\label{imageltwo}
  A morphism in $\Hilb$ is a partial isometry if
  and only if it is of the form $v \after \ltwo f \after u$ for morphisms $f$ in
  $\PInj$ and unitaries $u,v$ in $\Hilbi$.  
\end{proposition}
\begin{proof}
  Clearly a map of the form $v \after \ltwo f \after u$ is a partial
  isometry. Conversely, suppose that $i \colon H \to K$ is a partial
  isometry. Choose an orthonormal basis $X \subseteq H$ for its
  initial space $\ker(i)^\perp$, and choose an orthonormal basis $X'
  \subseteq H$ for $\ker(i)$, giving a unitary $u \colon H \to \ltwo(X \oplus X')$.
  Let $Y=i(X) \subseteq K$. Then $Y$ will be an orthonormal basis for
  the final space $\ker(i^\dag)^\perp$ because $i$ acts isometrically
  on $X$. Choose an orthonormal basis $Y' \subseteq K$ for
  $\ker(i^\dag)$, giving a unitary $v \colon \ltwo(Y \oplus Y') \to
  K$. Now, if we define $f = (\xyline{X \oplus X' & X \ar@{ >->}[l] 
  \ar@{ >->}|-{i}[r] & Y \oplus Y'})$, then $i=v \after \ltwo f \after u$.
\end{proof}

\begin{note}
\label{compositionpartialisometries}
  However, partial isometries are not closed under composition. To see
  this, consider the partial isometries $\left(\begin{smallmatrix} 1
  \\ 0 \end{smallmatrix}\right) \colon \field{C} \to \field{C}^2$
  and $\left(\begin{smallmatrix} \sin(\theta) &
  \cos(\theta) \end{smallmatrix}\right) \colon \field{C}^2 \to
  \field{C}$ for a fixed real number $\theta$. Their composition is
  $\left(\begin{smallmatrix} \sin(\theta) \end{smallmatrix}\right)
  \colon \field{C} \to \field{C}$, which is not a partial isometry unless
  $\theta$ is a multiple of $\pi/2$. 
  There are other compositions
  that do make partial isometries into a
  category~\cite{hinesbraunstein:partialisometries}, but these are not
  of interest here. Instead, we shall extend the previous proposition to
  highlight one of the most remarkable features of the functor
  $\ltwo$.
\end{note}

\begin{note}
\label{finiterank}
  The example in~\ref{compositionpartialisometries} shows that any
  linear function $\field{C} \to \field{C}$ between -1 and 1 is a
  composition of partial isometries. Note that the
  projections $\pi_i \colon \field{C}^m \to \field{C}$ and
  coprojections $\pi_i^\dag \colon \field{C} \to \field{C}^n$ are
  partial isometries, as are the weighted diagonal $\Delta / \sqrt{n} \colon \field{C} \to
  \field{C}^n$ given by $\Delta(x) = (x,\ldots,x)$ and the
  weighted codiagonal $\Delta^\dag  / \sqrt{m} \colon \field{C}^m \to \field{C}$ given by
  $\Delta^\dag(x_1,\ldots,x_m) = \sum_i x_i$. Moreover, it is easy to
  see that if $f$ and $g$ are (compositions of) partial isometries,
  then so is $f \oplus g$. Finally, any linear map $f \colon
  \field{C}^m \to \field{C}^n$ has a matrix expansion, and can hence
  be written in terms of biproduct structure as $f = \Delta^\dag
  \after (\bigoplus_{i=1}^m\bigoplus_{j=1}^n \pi_j^\dag \after \pi_j
  \after f \after \pi_i^\dag \after \pi_i) \after \Delta$. Thus any
  $f \colon \field{C}^m \to \field{C}^n$ with $\|f\| \leq 1/\sqrt{mn}$
  is a composition of partial isometries. 
\end{note}

\begin{note}
  The \emph{essential image} of a functor $F \colon
  \cat{C} \to \cat{D}$ is the smallest subcategory of $\cat{D}$
  that contains all morphisms $F(f)$
  for $f$ in $\cat{C}$, and that is closed under composition with
  bimorphisms of $\cat{D}$.

  It follows from~\ref{finiterank} that the essential
  image of the functor $\ltwo$ contains at least all morphisms of
  $\Hilb$ of finite rank. For infinite rank that strategy fails
  because $\Delta$ is then no longer a valid morphism
  (see~\ref{hilbconditionallycomplete}). 
  Nevertheless, Theorem~\ref{ltwoisessentiallyfull} below will prove
  that the essential image of $\ltwo$ is all of $\Hilb$.
  In preparation we accommodate an intermezzo on \emph{polar
  decomposition}. 

  A morphism $p \colon H \to H$ in $\Hilb$ is
  \emph{nonnegative} when $\inprod{px}{x} \geq 0$ for all $x \in H$, and
  \emph{positive} when $\inprod{px}{x}>0$. 
  Nonnegative maps are precisely those of the
  form $p=f^\dag f$ for some morphism $f$.
\end{note}

\begin{proposition}
\label{polar}
  For every morphism $f \colon H \to K$ between Hilbert spaces, there
  exist a unique nonnegative map $p \colon H \to H$ and partial isometry
  $i \colon H \to K$ satisfying $f=ip$ and $\ker(p)=\ker(i)$. 
\end{proposition}
\begin{proof}
  See~\cite[problem~134]{halmos:problembook}.
\end{proof}

\begin{note}
\label{strongpolar}
  The previous proposition stated the usual formulation of polar
  decomposition, but the unicity condition $\ker(p)=\ker(i)$ 
  is something of a red herring. It should be
  understood as saying that both $i$ and $p$ are uniquely determined
  on the orthogonal complement of $\ker(f)=\ker(p)=\ker(i)$. On each
  point of $\ker(f)$, one of $i$ and $p$ must be zero, but the other's
  behaviour has no restrictions apart from being a partial isometry or
  positive map, respectively. Dropping the unicity condition, we may
  take $p$ to be a positive map, by altering $i$ to be zero on $\ker(f)$,
  and $p$ to be nonzero on $\ker(f)$. More precisely, define $p'=p$ on
  $\ker(f)^\perp$ and $p'=\id$ on $\ker(f)$; since $\ker(f)$ is a
  closed subspace, $H \cong \ker(f) \oplus \ker(f)^\perp$, and this
  gives a well-defined positive operator $p' 
  \colon H \to H$. Similarly, setting $i'=i$ on $\ker(f)^\perp$ and
  $i'=0$ on $\ker(f)$ gives a well-defined partial isometry $i' \colon
  H \to K$, satisfying $f=i' p'$.

  In the following we have to consider bimorphisms rather than isomorphism because of the issues in~\ref{note:bimorphism}. In the finite-dimensional case, bimorphisms are isomorphisms.
\end{note}

\begin{lemma}
\label{positiveiso}
  Positive operators on Hilbert spaces are bimorphisms.
\end{lemma}
\begin{proof}
  Let $p \colon H \to H$ be a positive operator in $\Hilb$. 
  If $p(x)=0$ then certainly $\inprod{p(x)}{x}=0$ which contradicts positivity. Hence $\ker(p)=0$, and so $p$ is monic.

  To see that $p$ is epic, suppose that $p \circ f = p \circ g$ for parallel morphisms $f,g$. Then $\inprod{p \circ (f-g)(x)}{x}=0$ for all $x$. By positivity, For each $x$ there is $p_x>0$ such that $p \circ (f-g)x = p_x \cdot (f-g)(x)$. Hence $\inprod{(f-g)(x)}{x}=0$ for all $x$, that is, $f=g$ and $p$ is epic.
\end{proof}

\begin{definition}
\label{essentiallyfull}
  A functor $F \colon \cat{C} \to \cat{D}$ is \emph{essentially full}
  when for each morphism $g$ in $\cat{D}$ there exist $f$ in $\cat{C}$
  and $u,v$ in $\cat{C}_{\cong}$ such that $g = v \after Ff \after u$.

  It follows that the essential image of such a functor is all of $\cat{D}$.
\end{definition}

\begin{theorem}
\label{ltwoisessentiallyfull}
  The functor $\ltwo \colon \PInj \to \Hilb$ is essentially full.
\end{theorem}
\begin{proof}
  Let $g$ be a morphism in $\Hilb$. By Proposition~\ref{polar}
  and~\ref{strongpolar}, we can write $g=pi$ for a positive morphism
  $p$ and a partial isometry $i$. Use Proposition~\ref{imageltwo} to
  decompose $i=v' \after \ltwo f \circ u$ for $f$ in $\PInj$ and
  unitaries $v',u$. Finally, Lemma~\ref{positiveiso} shows that $v=p
  \after v'$ in $\Hilbi$ satisfies $g=v \circ \ltwo f \circ u$.
\end{proof}

\begin{note}
\label{essentiallyfullnotreducible}
  Writing $\two$ for the ordinal $2=(0 \leq 1)$ regarded as a
  category, the category $\cat{C}^\two$ is the \emph{arrow category} of
  $\cat{C}$: its objects are morphisms of $\cat{C}$, and its morphisms
  are pairs of morphisms of $\cat{C}$ making the square commute.
  A functor $F \colon \cat{C} \to \cat{D}$ is essentially
  full when $F^\two \colon \cat{C}^\two \to \cat{D}^\two$ is
  essentially surjective on objects. From this point of view
  Definition~\ref{essentiallyfull} is quite natural. Nonetheless we
  might consider weakening it to take $u=\id$ or
  $v=\id$. But this would break the previous theorem.
  For example, if $g \colon \ltwo(X) \to \ltwo(Y)$ is a morphism in
  $\Hilb$, there need not be $f \colon X \to Y$ in $\cat{PInj}$ and
  $v$ in $\Hilbi$ with $g = v \circ \ltwo f$. For a
  counterexample, take $X=Y=\{a,b\}$, and $g(a)=g(b)=a$; if $g=v \circ
  \ltwo f$, then $(v \circ \ltwo f)(a) = (v \circ \ltwo f)(b)$, so
  $(\ltwo f)(a)=(\ltwo f)(b)$, so $f(a)=f(b)$, whence $f$ cannot be a
  partial injection.  
  Similarly, because of the dagger, if $g \colon \ltwo(X) \to
  \ltwo(Z)$ is a morphism in 
  $\Hilb$, there need not be $f \colon Y \to Z$ in $\PInj$ and $u$ in
  $\Hilbi$ with $g = \ltwo f \circ u$. 
\end{note}

\section{The future}\label{sec:conclusion}

\begin{note}
\label{reconstruction}
  Theorem~\ref{ltwoisessentiallyfull} naturally raises a coherence
  question: is there any regularity to the bimorphisms $u$ and $v$ that
  enable us to write an arbitrary morphism of $\Hilb$ in the form $v
  \after \ltwo f \after u$? How do they behave under composition?
  Curiously enough, essentially full functors do not seem to have been
  studied in the categorical literature at all. The results in this
  article suggest such a study.

  It would be very interesting to reconstruct $\Hilb$ (up to
  equivalence) from $\Hilbi$ and $\PInj$ via the $\ltwo$ functor.
  The objects are easily recovered, because they are the same as those
  of $\Hilbi$. Theorem~\ref{ltwoisessentiallyfull} also lets us
  recover the homsets and identities, as soon as we can identify when two
  morphisms in $\Hilb$ of the form $v \after \ltwo f \after u$ are
  equal. The main problem is how to recover composition, which 
  requires a way to turn $\ltwo g \after v \after \ltwo f$ into $w
  \after \ltwo h \after u$. (Note that turning $\ltwo g \after v$
  into $w \circ \ltwo h$ would be sufficient, because we could then
  use functoriality of $\ltwo$ and composition in
  $\PInj$. But~\ref{essentiallyfullnotreducible} obstructs this; the
  bimorphism $v$ in the middle is crucial.) 
  This will likely lead into bicategorical territory.
\end{note}

\begin{note}
  The $\ltwo$--construction has a continuous counterpart, that turns
  a measure space $(X,\mu)$ into a Hilbert space $L^2(X,\mu)$ of
  square integrable complex functions on $X$. The $L^2$--construction
  is quite fundamental and well-studied, but surprisingly enough
  functorial aspects seem not to have been considered before.
  One possibility is to mimic Definition~\ref{def:ltwo}, and endow the
  category of measure spaces with essential injections $(X,\mu) \to
  (Y,\nu)$ as morphisms, \ie subsets $R \subseteq X \times Y$ such
  that $\nu(\{ y \mid xRy\})=0$ for all $x \in X$ and $\mu(\{x \mid
  xRy \})=0$ for all $y \in Y$.

  The importance of $L^2$--spaces lies in the following formulation of
  the spectral theorem: every normal operator $f \colon H \to H$
  is of the form $f=u^{-1} \circ g \circ u$ for a unitary $u \colon H
  \to L^2(X,\mu)$ and an operator $g$ induced by multiplication with a
  measurable function $X \to \field{C}$. This perspective warrants
  choosing complex measurable functions as (endo)morphisms on
  measure spaces, with multiplication for
  composition. With~\ref{strongpolar} in mind, we could even
  restrict to a groupoid of positive maps.
  A solution to~\ref{reconstruction}
  could then be regarded as reconstructing quantum mechanics (as
  embodied by $\Hilb$) from its continuous, quantitative aspects
  (encoded by the $L^2$ functor), and its discrete,
  qualitative aspects (encoded by the $\ltwo$ functor).

  At any rate, the continuous cousin $L^2$ of $\ltwo$ poses
  an interesting research topic. 
\end{note}

\begin{note}
  Letting $\mathcal{L}$ be the class of positive 
  morphisms, and $\mathcal{R}$ the class of partial isometries in
  $\Hilb$:
  \begin{enumerate}
  \item every morphism $f$ can be factored as $f=rl$ with $l \in
    \mathcal{L}$ and $r \in \mathcal{R}$;
  \item every commutative square as below with $l \in \mathcal{L}$ and
    $r \in \mathcal{R}$ allows a unique diagonal fill-in $d$ making
    both triangles commute.
    \[\xymatrix{
      \cdot \ar[r] \ar_-{l}[d] & \cdot \ar^-{r}[d] \\
      \cdot \ar[r] \ar@{-->}^-{d}[ur] & \cdot
    }\]
  \end{enumerate}
  The second property would follows from Lemma~\ref{positiveiso} for isomorphisms $l$. The established notion of
  \emph{orthogonal factorization system} additionally demands that (3) both
  $\mathcal{L}$ and $\mathcal{R}$ are closed under composition, and
  (4) all isomorphisms are in both $\mathcal{L}$ and
  $\mathcal{R}$. But (3) is not satisfied
  by~\ref{compositionpartialisometries}, and the map $-1 \colon H \to 
  H$ is a counterexample to (4). 

  Write $\three$ for the ordinal $3=(0 \leq 1 \leq 2)$, regarded as a
  category. Then objects of $\cat{C}^\three$ are composable pairs of morphisms.
  Recall that a \emph{functorial factorization} is a functor $F \colon
  \cat{C}^\two \to \cat{C}^\three$ that splits the composition
  functor. Lemma~\ref{positiveiso} ensures that polar decomposition
  at least provides a functorial factorization system.
  It is usual to require extra conditions on top of a functorial
  factorization, such as in a \emph{natural weak factorization
    system}. For details we refer
  to~\cite{grandistholen:naturalfactorizationsystems}. It leads too
  far afield here, but polar decomposition does not satisfy the axioms
  of a natural weak factorization system.

  In short, polar decomposition unquestionably provides a notion of
  factorization. But it does not fit existing categorical
  notions, despite the fact that factorization has been a topic of quite intense 
  study in category theory~\cite{freydkelly:factorisation,
    bousfield:factorization, grandistholen:naturalfactorizationsystems,
    korostenskitholen:factorization, rosickytholen:laxfactorization}.
  This is an interesting topic for further investigation.
\end{note}

\bibliographystyle{plain}
\bibliography{ltwo}

\begin{thebibliography}{10}

\bibitem{abramsky:retracing}
Samson Abramsky.
\newblock Retracing some paths in process algebra.
\newblock In {\em 7th International Conference on Concurrency Theory}, pages
  1--17. Springer, 1996.

\bibitem{abramskyblutepanangaden:nuclearideals}
Samson Abramsky, Richard Blute, and Prakhash Panangaden.
\newblock Nuclear and trace ideals in tensored *-categories.
\newblock {\em Journal of Pure and Applied Algebra}, 143:3--47, 1999.

\bibitem{abramskyheunen:hstar}
Samson Abramsky and Chris Heunen.
\newblock H*-algebras and nonunital {F}robenius algebras: first steps in
  infinite-dimensional categorical quantum mechanics.
\newblock In Samson Abramsky and Michael Mislove, editors, {\em Clifford
  Lectures}, volume~71 of {\em Proceedings of Symposia in Applied Mathematics},
  pages 1--24. American Mathematical Society, 2012.

\bibitem{abramskyjung:domaintheory}
Samson Abramsky and Achim Jung.
\newblock Domain theory.
\newblock In {\em Handbook of Logic in Computer Science Volume 3}, pages
  1--168. Oxford University Press, 1994.

\bibitem{adamekrosicky:locallypresentable}
Ji{\v{r}}{\'{\i}} Ad{\'a}mek and Ji{\v{r}}{\'{\i}} Rosi{c}ky.
\newblock {\em Locally Presentable and Accessible Categories}.
\newblock Number 189 in London Mathematical Society Lecture Note Series.
  Cambridge University Press, 1994.

\bibitem{barr:algebraicallycompact}
Michael Barr.
\newblock Algebraically compact functors.
\newblock {\em Journal of Pure and Applied Algebra}, 82:211--231, 1992.

\bibitem{bousfield:factorization}
Aldridge~K. Bousfield.
\newblock Constructions of factorization systems in categories.
\newblock {\em Journal of Pure and Applied Algebra}, 9(2--3):207--220, 1977.

\bibitem{cockettlack:restriction1}
J.~Robin~B. Cockett and Stephen Lack.
\newblock Restriction categories {I}: categories of partial maps.
\newblock {\em Theoretical Computer Science}, 270(1--2):223--259, 2002.

\bibitem{danosregnier:goi}
Vincent Danos and Laurent Regnier.
\newblock Proof-nets and the {H}ilbert space.
\newblock In {\em Advances in Linear Logic}, pages 307--328. Cambridge
  University Press, 1995.

\bibitem{freydkelly:factorisation}
Peter Freyd and Max Kelly.
\newblock Categories of continuous functors {I}.
\newblock {\em Journal of Pure and Applied Algebra}, 2, 1972.

\bibitem{grandistholen:naturalfactorizationsystems}
Marco Grandis and Walther Tholen.
\newblock Natural weak factorization systems.
\newblock {\em Archivum Mathematicum}, 42:397--408, 2006.

\bibitem{haghverdi:phd}
Esfandiar Haghverdi.
\newblock {\em A categorical approach to linear logic, geometry of proofs and
  full completeness}.
\newblock PhD thesis, University of Ottawa, 2000.

\bibitem{haghverdiscott:goi}
Esfandiar Haghverdi and Phil Scott.
\newblock A categorical model for the geometry of interaction.
\newblock {\em Theoretical Computer Science}, 350:252--274, 2006.

\bibitem{halmos:problembook}
Paul Halmos.
\newblock {\em A {H}ilbert space problem book}.
\newblock Springer, 2nd edition, 1982.

\bibitem{heunen:embedding}
Chris Heunen.
\newblock An embedding theorem for {H}ilbert categories.
\newblock {\em Theory and Applications of Categories}, 22(13):321--344, 2009.

\bibitem{heunenjacobs:dagkercat}
Chris Heunen and Bart Jacobs.
\newblock Quantum logic in dagger kernel categories.
\newblock {\em Order}, 27(2):177--212, 2010.

\bibitem{hines:phd}
Peter Hines.
\newblock {\em The algebra of self-similarity and its applications}.
\newblock PhD thesis, University of Wales, 1997.

\bibitem{hines:oracle}
Peter Hines.
\newblock Quantum circuit oracles for abstract machine computations.
\newblock {\em Theoretical Computer Science}, 411(11-13):1501--1520, 2010.

\bibitem{hinesbraunstein:partialisometries}
Peter Hines and Samuel~L. Braunstein.
\newblock {\em Semantic techniques in quantum computation}, chapter The
  structure of partial isometries, pages 361--389.
\newblock Cambridge University Press, 2009.

\bibitem{kadisonringrose:operatoralgebras}
Richard~V. Kadison and John~R. Ringrose.
\newblock {\em Fundamentals of the theory of operator algebras}.
\newblock Academic Press, 1983.

\bibitem{kastl:inverse}
J.~Kastl.
\newblock {\em Algebraische {M}odelle, {K}ategorien und {G}ruppoide}, volume~7
  of {\em Studien zur {A}lgebra und ihre {A}nwendungen}, chapter Inverse
  categories, pages 51--60.
\newblock Akademie-{V}erlag {B}erlin, 1979.

\bibitem{korostenskitholen:factorization}
Mareli Korostenski and Walter Tholen.
\newblock Factorization systems as {E}ilenberg-{M}oore algebras.
\newblock {\em Journal of Pure and Applied Algebra}, 85(1):57--72, 1993.

\bibitem{lawson:inversesemigroups}
Mark~V. Lawson.
\newblock {\em Inverse semigroups: the theory of partial symmetries}.
\newblock World Scientific, 1998.

\bibitem{maclane:categories}
Saunders {Mac Lane}.
\newblock {\em Categories for the Working Mathematician}.
\newblock Springer, 2nd edition, 1971.

\bibitem{reedsimon:functionalanalysis}
Michael Reed and Barry Simon.
\newblock {\em Methods of Modern Mathematical Physics, Vol {I}: Functional
  Analysis}.
\newblock Academic Press, 1972.

\bibitem{rosickytholen:laxfactorization}
Ji{\v{r}}{\'{\i}} Rosi{c}ky and Walter Tholen.
\newblock Lax factorization algebras.
\newblock {\em Journal of Pure and Applied Algebra}, 175:355--382, 2002.

\end{thebibliography}

\end{document}